\documentclass[12pt]{amsart}

\usepackage{amsfonts,amssymb,stmaryrd,amscd,amsmath,latexsym,amsbsy}

\newtheorem{theorem}{Theorem}[section]
\newtheorem{lemma}[theorem]{Lemma}
\newtheorem{proposition}[theorem]{Proposition}

\theoremstyle{definition}
\newtheorem{definition}[theorem]{Definition}

\newtheorem{example}[theorem]{Example}
\newtheorem{question}[theorem]{Question}

\newtheorem{remark}[theorem]{Remark}

\newcommand{\Tr}{\text{Tr}}

\newcommand{\Hom}{\text{Hom}}

\newcommand{\Rep}{\text{Rep}}

\newcommand{\Dim}{{\rm Dim}}
\newcommand{\kk}{{\bold k}}

\newcommand{\la}{\lambda}

\newcommand{\C}{\mathcal{C}}

\newcommand{\ben}{\begin{enumerate}}
\newcommand{\een}{\end{enumerate}}

\theoremstyle{plain}

\newtheorem*{sol}{Solution}

\theoremstyle{definition}

\theoremstyle{remark}

\newcommand{\solu}[1]{\begin{sol}{\bf (\ref{#1})}}

\begin{document}

\title{$p$-adic dimensions in symmetric tensor categories in characteristic $p$}

\author{Pavel Etingof, Nate Harman, Victor Ostrik}

\address{Etingof: Department of Mathematics, Massachusetts Institute of Technology, Cambridge, MA 02139,
USA}

\email{etingof@math.mit.edu}

\address{Harman: Department of Mathematics, Massachusetts Institute of Technology, Cambridge, MA 02139, USA}

\email{nharman@mit.edu}

\address{Ostrik: Department of Mathematics
University of Oregon
Eugene, OR 97403, USA}

\email{vostrik@math.uoregon.edu}

\begin{abstract} 
To every object $X$ of a symmetric tensor category over a field of characteristic $p>0$ 
we attach $p$-adic integers $\Dim_+(X)$ and $\Dim_-(X)$ whose reduction modulo $p$ 
is the categorical dimension $\dim(X)$ of $X$, coinciding with the usual dimension when $X$ is a vector space. 
We study properties of $\Dim_{\pm}(X)$, and in particular show that they don't always coincide with each other, 
and can take any value in $\Bbb Z_p$. We also discuss the connection of $p$-adic dimensions with the
theory of $\lambda$-rings and Brauer characters. 
\end{abstract}

\maketitle

\section{Introduction}

Let $\kk$ be an algebraically closed field of characteristic $p>0$, and let $\C$ be a symmetric tensor (in particular, abelian) 
category over $\kk$. Let $X\in \C$. Then to $X$ one can attach its categorical dimension $\dim(X)\in \kk$, 
and it is shown in \cite{EGNO}, Exercise 9.9.9(ii) that in fact $\dim(X)\in \Bbb F_p\subset \kk$. 
The main result of the paper is that $\dim(X)$ is the reduction modulo $p$ of a richer invariant of $X$, which takes values in 
$\Bbb Z_p$ --- the $p$-adic dimension of $X$. In fact, we define two kinds of $p$-adic dimensions --- the symmetric $p$-adic dimension 
$\Dim_+(X)$ and the exterior $p$-adic dimension $\Dim_-(X)$, which for vector spaces coincide with the usual dimension of $X$. 

Specifically, recall that if $t=\dots t_2t_1t_0\in {\Bbb Z}_p$ is a $p$-adic integer with digits $t_0,t_1,t_2,\dots$
(i.e., $t=\sum_{j\ge 0} t_jp^j$), then we can define the power series 
$$
(1+z)^t:=\prod_{j=0}^\infty (1+z^{p^j})^{t_j}\in {\Bbb F}_p[[z]],
$$
and $t$ is uniquely determined by the series $(1+z)^t$. 
This series has all the properties of the exponential function, i.e.,
$(1+z)^{t+s}=(1+z)^t(1+z)^s$, $((1+z)^t)^s=(1+z)^{ts}$, etc.

Now, the $p$-adic dimensions of $X$ are defined by the formulas
\begin{equation}\label{symdim}
\sum_{j=0}^\infty \dim(S^jX) z^j=(1-z)^{-\Dim_+(X)},
\end{equation}
and 
\begin{equation}\label{extdim} 
\sum_{j=0}^\infty \dim(\wedge^jX) z^j=(1+z)^{\Dim_-(X)},
\end{equation} 
where $S^iX$ and $\wedge^iX$ are the symmetric and exterior powers of $X$. 
The main challenge is to show that $\Dim_{\pm}(X)$ exist. We do it in Section 2.  
    
It follows immediately from the definition that 
$$
\Dim_\pm(X\oplus Y)=\Dim_\pm(X)+\Dim_\pm(Y)
$$ 
and that $p$-adic dimension is functorial,
i.e. for any (exact) symmetric tensor functor $F$ we have $\Dim_{\pm}(F(X))=\Dim_{\pm}(X)$.
Also, using Deligne's categories in characteristic $p$, we show that $\Dim_+(X)$ and $\Dim_-(X)$ 
can take any value in $\Bbb Z_p$. On the other hand, we show that in general 
$$
\Dim_+(X)\ne \Dim_-(X), 
\Dim_\pm(X\otimes Y)\ne \Dim_\pm(X)\Dim_\pm(Y),
$$ 
and at least in characteristic $2$ the dimensions $\Dim_\pm$ are not always additive on exact sequences, although these properties do 
hold for ``good'' objects. 

We will discuss applications of $p$-adic dimensions in subsequent papers. One of the motivations is to formulate 
the universal property of Deligne's categories $\Rep(S_t)$ and $\Rep(GL_t)$ in positive characteristic, defined by P. Deligne 
in his letter to the third author, \cite{De3} (see \cite{Ha}, Section 3.3, for a summary). Namely, for any nontrivial ultrafilter $\mathcal{U}$ 
on the set of natural numbers Deligne constructed the tensor categories of superexponential growth $\Rep(S_{\mathcal{U}})$, 
$\Rep(GL_{\mathcal{U}})$ over $\kk$, analogous to Deligne's categories $\Rep(S_t)$, $\Rep(GL_t)$ in characteristic zero 
(see \cite{DM,De1,De2} and \cite{EGNO}, Section 9.12). These categories are obtained by``$p$-adic interpolation'' of the ordinary representation categories 
of $S_n$ and $GL_n$ along ${\mathcal{U}}$. Deligne conjectured that they depend only on the $p$-adic integer $t:=\lim_{\mathcal{U}} n$
(so that one may denote them by $\Rep(S_t), \Rep(GL_t)$), and for $\Rep(S_{\mathcal{U}})$ this has now been proved in \cite{Ha} for $p\ge 5$. In characteristic zero, the categories $\Rep(S_t)$ and $\Rep(GL_t)$ have nice universal properties; for instance, symmetric tensor functors 
from $\Rep(GL_t)$, $\Rep(S_t)$ to a symmetric tensor category $\C$ correspond to objects (respectively, commutative Frobenius algebras) $X$ in $\C$ of dimension $t$ such that for any partition $\lambda$, the Schur functor ${\mathcal{S}}^\lambda X$ is nonzero. The notion of $p$-adic dimension should allow one to generalize these properties to positive characteristic; namely, one should expect that symmetric tensor functors from $\Rep(GL_t)$, $\Rep(S_t)$ to 
$\C$ correspond to ``good" objects (respectively, commutative Frobenius algebras) $X$ in $\C$ of {\it $p$-adic dimension $t$}. We plan to discuss these universal properties more precisely elsewhere. 

The organization of the paper is as follows. In Section 2 we define the $p$-adic dimensions and prove their existence. 
In Section 3 we study the properties of $p$-adic dimensions, and discuss their connection with $\lambda$-rings and Brauer characters. 

{\bf Acknowledgements.} We are very grateful to 
Pierre Deligne for his letter to V.O., which led to this work. 
We thank Haynes Miller for Remark \ref{Miller} and  
Siddharth Venkatesh for useful discussions. 
The work of P.E. and N.H. was partially supported by the NSF grant DMS-1502244. The work of N. H. was also partially supported by the 
National Science Foundation Graduate Research Fellowship 
under Grant No. 1122374. 
\section{$p$-adic dimensions} 

\subsection{Preliminaries on tensor categories} 

 Let $\kk$ be an algebraically closed field of characteristic $p>0$, and $\C$ a symmetric tensor category over $\kk$ 
(see \cite{EGNO}, Definitions 4.1.1, 8.1.2). For $X,Y\in \C$, we have the commutativity constraint 
$c_{X,Y}: X\otimes Y\to Y\otimes X$. Since $c_{Y,X}\circ c_{X,Y}=1_{X\otimes Y}$, we have a natural action of 
the symmetric group $S_n$ on $X^{\otimes n}$ defined by $s_{i,i+1}\mapsto c_i:=1^{\otimes i-1}\otimes c_{X,X}\otimes 1^{n-i-1}$,
where $s_{i,j}\in S_n$ is the transposition of $i$ and $j$.  

Recall from \cite{EGNO} that to every endomorphism $A:X\to X$ of an object $X\in \C$ we can attach its trace
$$
\Tr_X(A):={\rm ev}_X\circ c_{X,X^*}\circ (A\otimes 1)\circ {\rm coev}_X\in {\rm End}(\bold 1)=\kk,
$$
and in particular we can define the dimension of $X$ by $\dim(X)=\Tr_X(1)\in \kk$. 
We have $\Tr_{X^*}(A^*)=\Tr_X(A)$, hence $\dim(X^*)=\dim(X)$. 

\begin{lemma}\label{tra} Let $\sigma\in S_n$ be a permutation. Then $\Tr_{X^{\otimes n}}(\sigma)=\dim(X)^{c(\sigma)}$, where $c(\sigma)$ is the number of 
cycles of $\sigma$. 
\end{lemma} 

\begin{proof} This is easy to prove diagrammatically. Consider the flat braid $b_\sigma$ corresponding to $\sigma$. 
Then the trace of $\sigma$ is depicted by the flat link obtained by gluing the $i$-th input of $b_\sigma$ to its $i$-th output for all $i\in [1,n]$. 
It is easy to see that this flat link has $c(\sigma)$ components, and each component contributes a factor of $\dim(X)$, which implies the statement.  
\end{proof} 

\begin{lemma}\label{dimfp} (\cite{EGNO}, Exercise 9.9.9(ii)) One has $\dim(X)\in \Bbb F_p$. 
\end{lemma} 

\begin{proof} Let $s$ be the cyclic permutation acting on $X^{\otimes p}$. Then $s^p=1$, so $(1-s)^p=0$ (as we are in characteristic $p$), and hence 
$1-s$ is nilpotent. Thus $\Tr_{X^{\otimes p}}(1-s)=0$. But by Lemma \ref{tra}, this trace equals $\dim(X)^p-\dim(X)$. Thus $\dim(X)^p=\dim(X)$,
so $\dim(X)\in \Bbb F_p$. 
\end{proof} 

 Now recall the definition of symmetric and exterior powers of $X$. 
 This definition is given in \cite{EGNO}, Section 9.9, under the assumption of characteristic zero;
 the discussion also applies to characteristic $p>2$, but some changes are needed for $p=2$. 

 Let $\Lambda^2X$ be the image of $c_{X,X}-1$ in $X\otimes X$. We define the symmetric algebra $SX$
 to be the quotient of the tensor algebra $TX$ by the two-sided ideal generated by $\Lambda^2X$. Then 
 $SX$ is a $\Bbb Z_+$-graded algebra: $SX=\oplus_{n\ge 0} S^nX$. The object $S^nX$ is the object of 
 coinvariants of $S_n$ in $X^{\otimes n}$, and is called {\it the $n$-th symmetric power} of $X$.
 We also define the {\it dual $n$-th symmetric power}  ${\Bbb S}^nX:=(S^nX^*)^*$. 
 We have a natural inclusion $\Bbb S^nX\subset X^{\otimes n}$ as the subobject of $S_n$-invariants; i.e., 
$\Bbb S^nX$ is the intersection of the kernels of $c_i-1$ for $i=1,...,n-1$. We also have a natural projection $X^{\otimes n}\to S^nX$. 
The composition of these two maps is a morphism $\phi_n: \Bbb S^nX\to S^nX$. This morphism is an isomorphism for $n<p$ but not for $n\ge p$. 

 Similarly, consider ${\Bbb S}^2X\subset X\otimes X$, and define the exterior algebra $\wedge X$ 
 as the quotient of $TX$ by the two-sided ideal generated by ${\Bbb S}^2X$. Then $\wedge X=\oplus_{n\ge 0}\wedge^nX$, and $\wedge^nX$ 
 is a quotient of $X^{\otimes n}$ called the {\it $n$-th exterior power of $X$}. We also define the {\it dual $n$-th exterior power} 
 $\Lambda^nX:=(\wedge^nX^*)^*\subset X^{\otimes n}$, which is the intersection of the images of $c_i-1$ for $i=1,...,n-1$
 (so for $n=2$ we recover $\Lambda^2X$ defined above). We have a morphism $\psi_n: \Lambda^nX\to \wedge^nX$ which is an isomorphism for $n<p$ but not for $n\ge p$.
 
If $p>2$ then $S^2X\cong \Bbb S^2X$, so the exterior power $\wedge^n X$ can be defined as the object of anticoinvariants 
 of $S_n$ in $X^{\otimes n}$, while the dual exterior power $\Lambda^n X\subset X^{\otimes n}$ is the subobject 
 of antiinvariants (the intersection of the kernels of $c_i+1$ or, equivalently, images of $c_i-1$ over all $i$). 

Also, in any characteristic these definitions of symmetric and exterior algebras and powers coincide with the usual ones if $X$ is a vector space over $\kk$. 

Finally, note that $\dim(S^iX)=\dim(\Bbb S^iX^*)$ and $\dim(\wedge^i X)=\dim(\Lambda^iX^*)$. 

\subsection{The $p$-adic dimensions} 

Let $X\in \C$. The following theorem is our main result. 

\begin{theorem}\label{pdimexists}  
(i) There exists a unique $d_+=d_+(X)\in \Bbb Z_p$ such that 
$$
\sum_{j\ge 0}\dim(\Bbb S^jX)z^j=
(1-z)^{-d_+}.
$$ 

(ii) There exists a unique $d_-=d_-(X)\in \Bbb Z_p$ such that 
$$
\sum_{j\ge 0}\dim(\Lambda^j X)z^j=
(1+z)^{d_-}.
$$ 
\end{theorem} 

\begin{proof} 
(i) Let $d_r:=\dim(\Bbb S^rX)$.
Let $n$ be a positive integer, and $\la=(\la_1,\la_2,...,\la_m)$ be a partition of $n$ of length $m$.
Let 
$$
S_\la:=S_{\la_1}\times...\times S_{\la_m}\subset S_n
$$ 
be the Young subgroup of $\lambda$.
Let 
$$
\Bbb S^\la X:=\Bbb S^{\la_1}X\otimes...\otimes \Bbb S^{\la_m}X=(X^{\otimes n})^{S_\la}.
$$  

Given a double coset $B\in S_\la\backslash S_n/S_\la$, 
we have a canonical endomorphism $A_B: \Bbb S^\la X\to \Bbb S^\la X$. 
Namely, write $B$ as a disjoint union of
left cosets $B_i$ of $S_\la$ in $S_n$, and pick elements $b_i\in B_i$.
Then define the endomorphism $A'$ of $X^{\otimes n}$ by $A':=\sum_i b_i$.

{\bf Claim 1.} $A'$ restricts to an endomorphism $A_B$ of $\Bbb S^\la X$,
and the resulting endomorphism is independent on the choice of the $b_i$.

Indeed, if $b_i'\in B_i$ is another choice, then $b_i'=b_ig_i$, where $g_i\in S_\la$, 
so $b_i'=b_i$ on $\Bbb S^\lambda X$. Also, for any $g\in S_\la$ we have 
$g\sum_ib_i=\sum_i gb_i$, and $gb_i$ is a representative of the coset $gB_i$. 
But the union of $gB_i$ over all $i$ is disjoint and equals $B$, so $\sum_i gb_i=\sum_i b_i$ on $\Bbb S^\lambda X$. 
This proves the claim. 

\begin{lemma}\label{le1}
 There exists a universal polynomial $P_B(x_1,x_2,...,x_{\la_1})$
with integer coefficients such that for a tensor category $\C$ over a field of any characteristic and any $X\in \C$ we have 
$\Tr_{\Bbb S^\la X}(A_B)=P_B(d_1,d_2,...d_{\la_1})$.
\end{lemma} 

\begin{proof} The proof is by reverse induction in the length $m$ of $\la$, starting with
length $n$. If ${\rm length}(\la)=n$ then $\Bbb S^\la X=X^{\otimes n}$, $S_\la=1$, $B\in S_n$, $A_B=B$, 
and $\Tr_{\Bbb S^\la X}(A_B)=d_1^{c(B)}$ by Lemma \ref{tra}, which gives the base of induction.
So assume that ${\rm length}(\la)=m<n$ and the statement is known for all larger lengths.
If $B$ is the double coset of $1$ then $A_B=1$, and $\Tr_{\Bbb S^\la X}(A_B)=\prod_{i=1}^m d_{\la_i}$, as required.
So we may assume without loss of generality that $B$ is the double coset of $g\in S_n$, where $g\notin S_\la$. 
Let $H\subset S_\la$ be the intersection of $S_\la$ with $gS_\la g^{-1}$. 
Recall that double cosets of $S_\la$ in $S_n$ are labeled by $m$ by $m$ matrices $\bold a=(a_{ij})$ 
of nonnegative integers such that $\sum_i a_{ij}=\la_i$ and $\sum_ja_{ij}=\la_j$ (see \cite{St}, Exercise 7.77). 
By choosing $g$ to be of minimal length, we may assume that 
$$
g^{-1}Hg=S_{\bold a}:=\prod_{i=1}^m\prod_{j=1}^m S_{a_{ij}}\subset S_\la,
$$
 where $(a_{ij})$ is a partition of $\la_i$ (nontrivial for at least one $i$); see 
\cite{S}, Lemma 2. 
 Let $c_s$ be representatives of elements $s$ of $S_\la/H$ in $S_\la$. Then we can write 
 $A_B$ as 
 $$
 A_B=\sum_{s\in S_\la/H} c_sg
 $$ 
 (note that this is independent on the choice of $c_s$ since $g^{-1}Hg\subset S_\la$). 
 
Let  $\Bbb S^{\bold a}X:=\otimes_i \otimes_j \Bbb S^{a_{ij}}X$, and let
$A': \Bbb S^{\bold a}X\to X^{\otimes n}$ be defined by the same formula as $A_B$, i.e.,
$$
A'=\sum_{s\in S_\la/H} c_sg.
$$

{\bf Claim 2.} $A'$ lands in $\Bbb S^\la X\subset X^{\otimes n}$.

Indeed, let $y\in S_\la$. Then 
$$
yA'=\sum_s yc_sg=\sum_s c_sb_sg=\sum_s c_sgb_s',
$$
where $b_s\in H$ and $b_s'=g^{-1}b_sg\in g^{-1}Hg=S_{\bold a}$. So $b_s'$ acts trivially on $\Bbb S^{\bold a}X$, hence $yA'=A'$, and thus $A'$ lands in $\Bbb S^\la X$, as claimed.

It is clear that $A'$ restricts on $\Bbb S^\la X$ to $A_B$. Thus, $\Tr_{\Bbb S^\la X}(A_B)=\Tr_{\Bbb S^{\bold a}X}(A')$. But $\Tr_{\Bbb S^{\bold a}X}(A')$ is given by some universal polynomial in $d_1,...,d_{\la_1}$ with integer coefficients by the induction assumption, as $A'$ is a sum of  endomorphisms of the form $A_{B'}$, where $B'$ is a double coset of $S_{\bold a}$ in $S_n$, and $S_{\bold a}$ is a parabolic subgroup corresponding to the partition $\bold a$ with ${\rm length}(\bold a)>m$. This proves the lemma. 
\end{proof} 

\begin{lemma}\label{le2} Suppose that the ground field has characteristic $p>0$. Then
for each $n$ there exists a universal polynomial $Q_n(z_0,z_1,...)$ over $\Bbb F_p$
such that $d_n=Q_n(d_1,d_p,d_{p^2},...)$. Moreover, $Q_n$ is unique if we
require that it is of degree $<p$ in every variable. This unique polynomial is given by the formula
\begin{equation}\label{qn}
Q_n(z_0,z_1,z_2,...)=(-1)^n\prod_{i=0}^k \binom{(-1)^{p^i}z_i}{n_i},
\end{equation} 
where $n_i$ are the digits of the base $p$ expansion of $n$.
\end{lemma} 

\begin{proof} To prove existence, use induction in $n$.
The base $n=1$ is trivial. We may assume that $n$ is not a power of $p$
(otherwise, if $n=p^i$, then $Q_n(z)=z_i$). Suppose the statement is known
below $n$. Let $p^i$ be the largest power of $p$ dividing $n$ ($i\ge 0$). Then it is easy to see that 
the binomial coefficient $\binom{n}{p^i}$ is not zero in $\Bbb F_p$.
Therefore, $\Bbb S^nX$ is the image of the following projector
$P$ acting on $\Bbb S^{n-p^i}X\otimes \Bbb S^{p^i}X$: 
$$
P=\binom{n}{p^i}^{-1}\sum_s c_s,
$$
where $c_s$ are representatives of the elements $s\in S_n/S_{n-p^i}\times S_{p^i}$ in $S_n$. Thus, $d_n=\Tr(P)$.
But by Lemma \ref{le1}, $\Tr(P)$ is given by a universal polynomial over $\Bbb F_p$ in $d_1,...,d_r$, where 
$$
r={\rm max}(n-p^i,p^i).
$$ 
Since $r<n$, by the induction assumption, $d_i$, $i\le r$ are universal polynomials over $\Bbb F_p$ of $d_1,d_p,d_{p^2},\dots $.
Hence so is  $d_n$. This shows that the polynomial $Q_n$ exists. 

Moreover, since by Lemma \ref{dimfp} we have $d_i\in \Bbb F_p$,
in $Q_n$ we can replace $z_i^p$ with $z_i$, and thus can choose $Q_n$ in such a way that it has degree $<p$ with respect to each variable. 

Now we want to show that $Q_n$ with this degree requirement is unique and is given by formula \eqref{qn}. 
To do so, pick some $Q_n$ satisfying the degree requirement, and let 
$$
f(z_0,...,z_k)=Q_n(z_0,...,z_k)-(-1)^n\prod_{i=0}^k \binom{(-1)^{p^i}z_i}{n_i},
$$
for sufficiently large $k$. 
Let $N_0,N_1,...,N_k\in [0,...,p-1]$, and let
$$
N=p^{k+1}-\sum_{i=0}^k N_ip^i.
$$
Let $\C$ be the category of vector spaces, and $X=\kk^N$. Then 
$d_r=\binom{N+r-1}{r}=(-1)^r\binom{-N}{r}$. By Lucas' theorem, for $r<p^{k+1}$ we then get 
$$
d_r=(-1)^r\prod_{i=0}^k \binom{N_i}{r_i},
$$
where $r_j$ are the base $p$ digits of $r$, and we view the right hand side modulo $p$. In particular, 
$d_{p^i}=(-1)^{p^i}N_i$ for $i=0,...,k$, where $N_i$ are taken mod $p$. Thus, we get 
$$
d_r=(-1)^r\prod_{i=0}^k \binom{(-1)^{p^i}d_{p^i}}{r_i},
$$
so $f(d_1,d_p,...,d_{p^k})=0$ for any $d_1,d_p...,d_{p^k}\in \Bbb F_p$ (as $N_i$ are arbitrary). 
Since the degree of $f$ in each variable is $<p$, this implies that $f=0$, as desired. 
\end{proof} 

Now we are ready to prove (i). Let $\delta_i$ be the representative of $(-1)^{p^i}d_{p^i}$ in 
$[0,p-1]$, and let $d_+=-\sum_{i\ge 0}\delta_i p^i\in \Bbb Z_p$. We have 
$$
\sum_{n=0}^\infty d_n z^n=\prod_{i\ge 0} \sum_{n_i=0}^{p-1}\binom{(-1)^{p^i}d_{p^i}}{n_i}(-z)^{n_ip^i}=
$$
$$
\prod_{i\ge 0} \sum_{n_i=0}^{p-1}\binom{\delta_i}{n_i}(-z)^{n_ip^i}=\prod_{i\ge 0} (1-z^{p^i})^{\delta_i}=(1-z)^{-d_+}.
$$
This proves (i). 

(ii) First assume that $p>2$. Consider the symmetric tensor category 
$\C':={\rm Supervec}\boxtimes \C$, and 
define the object $\Pi X:=X\otimes \kk_-\in \C'$, where $\kk_-$ is the 
odd 1-dimensional supervector 
space. Then $\Lambda^n X=\Bbb S^n\Pi X$ for even $n$ 
 and $\Lambda^n X=\Pi \Bbb S^n\Pi X$ for odd $n$. Thus, (ii) reduces to (i), with
$d_-(X)=-d_+(\Pi X)$. 

Now consider the case $p=2$. In this case, we will use an argument 
analogous to the proof of (i), replacing $\Bbb S^\bullet$ by $\Lambda^\bullet$
(this argument can also be used as an alternative proof for $p>2$, inserting appropriate signs). 
Let us describe the necessary changes. First we need to prove an analog of Claim 1: 

{\bf Claim 3.} $A'$ restricts to an endomorphism $A_B$ of $\Lambda^\la X$,
and the resulting endomorphism is independent on the choice of the $b_i$.

To see this, note first that $\Lambda^\lambda X\subset \Bbb S^\lambda X$, 
so by Claim 1, the choice of $b_i$ does not matter. To show that 
$A_B$ lands in $\Lambda^\lambda X$, we need to show that 
$A_B$ lands in the image of $c_j+1$ in $X^{\otimes n}$ for each $j\ne \lambda_1,\lambda_1+\lambda_2,...$. 
Let $T_j$ be the set of all $i$ such that $s_{j,j+1}B_i=B_i$.
If $i\notin T_j$, we may choose $b_i\in B_i$ and $b_{i'}\in B_{i'}:=s_{j,j+1}B_i$ so that 
$s_{j,j+1}b_i=b_{i'}$. Thus, it suffices to show that for each $i\in T_j$, $b_i$ maps  
$\Lambda^\lambda X$ to the image of $c_j+1$. To this end, note that $s_{j,j+1}b_i=b_is_{k,l}$, 
where $s_{k,l}\in S_\la$. Since $\wedge^\la X$ is contained in the image of $s_{k,l}+1$, the image 
of $b_i$ restricted to $\Lambda^\la X$ is contained in the image 
of $b_i(s_{k,l}+1)=(s_{j,j+1}+1)b_i$ on $X^{\otimes n}$. This implies Claim 3. 

Now we need to prove an analog of Claim 2. 

{\bf Claim 4.} $A'$ lands in $\Lambda^\la X\subset X^{\otimes n}$.

To see this, let $U_j$ be the set of $s$ such that $s_{j,j+1}s=s$. As in Claim 3, 
it suffices to show that for $s\in U_j$, $c_sg$ restricted to $\Lambda^{\bold a}X$ lands in 
the image of $s_{j,j+1}+1$. We have $s_{j,j+1}c_sg=c_sgs_{k,l}$, where $s_{k,l}\in S_{\bold a}$. 
We have $\Lambda^{\bold a} X\subset {\rm Im}(s_{k,l}+1)$, so the image of the 
restriction of $c_sg$ to $\Lambda^{\bold a} X$ is contained in the image of
$c_sg(s_{k,l}+1)=(s_{j,j+1}+1)c_sg$. This proves Claim 4. 

The rest of the proof is a straightforward generalization of the proof of (i), starting from Lemma \ref{le2},
and the formulas are simpler. Namely, we have $N=N_0+N_1p+...+N_kp^k$, and 
the universal polynomial analogous to the polynomial $Q_n$ of Lemma \ref{le2} has the form
$$
Q_n^-(z_0,z_1,...)=\prod_{i\ge 0}\binom{z_i}{n_i}.
$$
\end{proof} 

\begin{example} Here are examples of computation of the polynomial 
$P_B$ constructed in the proof of Lemma \ref{le1}. 

1. Let $n=3$, $\lambda=(2,1)$. Then $S_n=S_3$ consists of two double cosets 
of $S_\lambda=S_2=\langle(12)\rangle$ --- the group $S_\lambda$ itself and the remaining set $B$ of $4$ elements: 
$B=\lbrace{(23),(13),(123),(132)\rbrace}$. The element of minimal length in $B$ is $g=(23)$ (it has length $1$). 
Then $H$ is the trivial group, and so is $S_{\bold a}$. Thus $A_B=(23)+(12)(23)=(23)+(123)$. 
Hence 
$$
P_B=\Tr(A_B)=\Tr(A')=\Tr_{X^{\otimes 3}}((23)+(123))=d_1^2+d_1.  
$$
So in characteristic $2$ we get $d_3=d_1d_2+d_1^2+d_1$, which gives $d_1d_2=Q_3(d_1,d_2)$ (as $d_1\in \Bbb F_2$). 

2. Let $n=4$, $\lambda=(3,1)$. Then $S_n=S_4$ consists of two double cosets 
of $S_\lambda=S_3=\langle (12),(23)\rangle$ --- the group $S_\lambda$ itself and the remaining set $B$ of $18$ elements. 
The element of minimal length in $B$ is $g=(34)$ (it has length $1$). Thus $H=\langle (12)\rangle=S_{\bold a}$.
Hence $A_B=(34)+(123)(34)+(213)(34)=(34)+(1234)+(2134)$. So 
$$
P_B=\Tr(A_B)=\Tr(A')=\Tr_{\Bbb S^2X\otimes X\otimes X}((34)+(1234)+(2134))=
$$
$$
=d_1d_2+\Tr_{\Bbb S^2X\otimes X\otimes X}((1234)+(234)).
$$ 
The remaining trace is the trace corresponding to a double coset $B'$ of $S_2=\langle (12)\rangle$ in $S_4$ consisting of the elements 
$(134),(234),(1234),(2134)$ in which the minimal length element is $(234)$ (of length $2$). 
So this trace equals 
$$
\Tr_{X^{\otimes 4}}((1234)+(234))=d_1+d_1^2.
$$
Thus we get that $P_B=d_1d_2+d_1+d_1^2$. Hence, in characteristic $3$ we have 
$$
d_4=d_1d_3+\frac{1}{2}d_1(d_1^2+d_1)+d_1+d_1^2=d_1d_3-d_1^3+d_1,
$$
which gives $d_1d_3=Q_3(d_1,d_3)$, as $d_1\in \Bbb F_3$. 
\end{example} 

\begin{definition}
We call the number $d_+(X^*)$ the {\it symmetric $p$-adic dimension of $X$}, denoted $\Dim_+(X)$, 
and call the number $d_-(X^*)$ the {\it exterior $p$-adic dimension} of $X$, denoted $\Dim_-(X)$. 
\end{definition} 

Thus, the numbers $\Dim_\pm(X)$ satisfy equations \eqref{symdim},
\eqref{extdim}. 

\begin{example} 
Let $X$ be a supervector space. Then $\Dim_+(X)=\Dim_-(X)=\dim(X_0)-\dim(X_1)$, where 
$X_0$ and $X_1$ are the even and odd parts of $X$. 
\end{example} 

Let $R$ be the ring of integer-valued polynomials, i.e. polynomials $f\in \Bbb Q[t]$ 
which map integers to integers (or, equivalently, map sufficiently large integers to integers). 
It is well known that $R$ has a $\Bbb Z$-basis consisting of the binomial coefficients 
$e_i:=\binom{t}{i}$, $i\ge 0$, with multiplication law 
$$
e_ie_j=\sum_k C_{ij}^k e_k, \quad C_{ij}^k=\frac{k!}{(k-i)!(k-j)!(i+j-k)!}.
$$
In particular, the leading coefficient (i.e., that of $e_{i+j}$) is $\binom{i+j}{i}$. 

Recall that if $m$ is not a power of a prime $p$ then there exists an integer $0<i<m$ such that $\binom{m}{i}$ 
is not divisible by $p$. Thus, $R$ is generated by $e_j$ where $j$ is a prime power. 

Let ${\mathcal{I}}\subset \Bbb Z[x_1,x_2,...]$ 
be the ideal of polynomials $f$ such that for a symmetric tensor category ${\mathcal{C}}$ over any field and any $X\in {\mathcal{C}}$,   
one has $f(d_1,d_2,...)=0$, where $d_i=\dim \wedge^i X$. Let ${\mathcal{R}}:=\Bbb Z[x_1,x_2,...]/{\mathcal{I}}$. 

\begin{proposition}\label{iso} We have an isomorphism $\psi: {\mathcal{R}}\to R$ given by $\psi(x_i)=e_i$. 
\end{proposition} 

\begin{proof} Consider the map $\xi: \Bbb Z[x_1,x_2,...]\to R$ sending $x_i$ to $e_i$. Let $I={\rm Ker}\xi$. 
We need to show that $I={\mathcal{I}}$. Note that $I$ is generated by the elements $f_{ij}:=x_ix_j-\sum_k C_{ij}^k x_k$.
We claim that $f_{ij}\in {\mathcal{I}}$, i.e. 
\begin{equation}\label{quadrel} 
d_id_j=\sum_k C_{ij}^k d_k.
\end{equation} 
in any tensor category. To see this, multiply the left hand side by $z^iw^j$ and sum over $i,j$. 
We get $(1+z)^d(1+w)^d$, where $d=\Dim_-(X)$. Write this as $(1+z+w+zw)^d$ and expand it as 
$$
(1+z)^d(1+w)^d=\sum_k d_k (z+w+zw)^k=
$$
$$
\sum_{i,j,k}\frac{k!}{(k-i)!(k-j)!(i+j-k)!}d_kz^{k-j}w^{k-i}(zw)^{i+j-k}=
$$
$$
\sum_{i,j,k}\frac{k!}{(k-i)!(k-j)!(i+j-k)!}d_kz^iw^j,
$$
which is exactly the generating function of the right hand sides of \eqref{quadrel}. Thus, $I\subset {\mathcal{I}}$, and our job is to 
establish the opposite inclusion. To do so, let $f\in {\mathcal{I}}$, and assume that $f$ depends 
only on $x_1,...,x_r$. Consider the category ${\mathcal C}={\rm Vect}_{\Bbb C}$ 
of complex vector spaces. Let $X=\Bbb C^n$. Then $d_i=\binom{n}{i}$, so $f(1,n,...,\binom{n}{r})=0$ for any positive integer $n$. 
Then $f(1,t,...,\binom{t}{r})=0$ in $R$, which implies that $f\in I$, as desired.   
\end{proof} 

\begin{remark} Using Proposition \ref{iso}, we see that the polynomial $P_B$ in 
Lemma \ref{le1} is not unique, since $R$ is generated by $e_j$ with $j$ being a prime power
(so, $d_6$ expresses in terms of $d_1$,...,$d_5$). 
\end{remark} 

\begin{remark} Let $\widetilde{\mathcal{I}}\subset \Bbb Z[x_1,x_2,...]$ 
be the ideal of polynomials $f$ such that for a symmetric tensor category ${\mathcal{C}}$ over any field and any $X\in {\mathcal{C}}$,   
one has $f(d_1,d_2,...)=0$, where $d_i=\dim S^i X$. Let $\widetilde{\mathcal{R}}:=\Bbb Z[x_1,x_2,...]/\widetilde{\mathcal{I}}$. 
Then similarly to Proposition \ref{iso} one shows that we have an isomorphism $\widetilde{\psi}: \widetilde{\mathcal{R}}\to R$
such that $\widetilde{\psi}(x_i)=s_i$, where $s_i:=\binom{t+n-1}{n}\in R$. 
\end{remark} 

\section{Properties of $p$-adic dimensions} 

In this section we will study properties of $p$-adic dimensions $\Dim_\pm$. 
Some of the properties are obvious from the definition --- e.g. $\Dim_\pm$ are invariant 
under symmetric tensor functors, and $\Dim_\pm (X)$ are equal to $\dim(X)$ modulo $p$. 
Some less obvious properties are discussed below. 

\subsection{Values of $p$-adic dimensions} 
\begin{proposition}\label{prope1}
$\Dim_+(X)$ and $\Dim_-(X)$ can take any value in $\Bbb Z_p$. 
\end{proposition} 

\begin{proof} 
This follows from the existence of Deligne's examples (see \cite{Ha}, Subsection 3.3).
Namely, e.g. take Deligne's category $\Rep(S_{\mathcal{U}})$ for an ultrafilter ${\mathcal{U}}$, and let $V$ be the tautological object in this category. 
Then it is easy to see that $\Dim_+(V)=\Dim_-(V)=t$, where $t=\lim_{\mathcal{U}}n\in \Bbb Z_p$, and $t$ can be arbitrary. 
\end{proof} 

\subsection{Additivity on exact sequences}

If $Y$ is a $\Bbb Z_+$-filtered object of a symmetric tensor category ${\mathcal{C}}$, then we will denote by 
${\rm gr}Y$ the associated graded object 
$\oplus_i F_{i+1}Y/F_iY$; so, $\dim Y=\dim {\rm gr}Y$. Also note that if 
$Y$ is $\Bbb Z_+$-filtered then the symmetric algebra $SY$ carries an induced filtration. 

 Let 
$$
0\to X\to Y\to Z\to 0
$$ 
be a short exact sequence; it induces a 2-step filtration on $Y$. 

\begin{proposition}\label{prope2}
Suppose that the natural surjective morphism 
$$
\phi_+: S(X\oplus Z)=S({\rm gr}Y)\to {\rm gr}SY
$$
is an isomorphism. Then $\Dim_+(Y)=\Dim_+(X)+\Dim_+(Z)$. 
Likewise, if the natural surjective morphism 
$$
\phi_-: \wedge(X\oplus Z)=\wedge({\rm gr}Y)\to {\rm gr}\wedge Y
$$
is an isomorphism then $\Dim_-(Y)=\Dim_-(X)+\Dim_-(Z)$.
\end{proposition} 

\begin{proof} For a graded ind-object $W=\oplus_{n\ge 0}W_n$ 
of ${\mathcal{C}}$, define its Hilbert series by 
$$
h_W(z)=\sum_{n\ge 0}\dim W_n z^n.
$$
Then by definition of $p$-adic dimensions, for any object $V\in {\mathcal{C}}$, 
$h_{SV}(z)=(1-z)^{-\Dim_+(V)}$ and $h_{\wedge V}(z)=(1+z)^{\Dim_-(V)}$. 
Thus, the result is obtained by computing the Hilbert series of $SY$ as the product of Hilbert series of $SX$ and $SZ$, using that ${\rm gr}SY=S(X\oplus Z)=SX\otimes SZ$, 
and similarly for exterior algebras.   
\end{proof} 

\begin{example}\label{chara2} The assumption in Proposition \ref{prope2} cannot be removed, and, in fact, its conclusion can fail 
at least in characteristic $2$. To give an example, let $D$ be the Hopf algebra $\kk[d]/d^2$ 
with $d$ being a primitive element and triangular structure $R=1\otimes 1+d\otimes d$. 
Consider the non-semisimple symmetric tensor category $\Rep D$.
This category plays the role of the category of supervector spaces in characteristic $2$. Take $Y=D$, and $X=Z=\bold 1$. Then the map 
$\phi_+$ is not an isomorphism; in fact, for $n\ge 1$, $S^nD=D$ for odd $n$ and $S^nD=\bold 1\oplus \bold 1$ for even $n$, since $SD=\bold k[x,y]/y^2$ with $dx=y,dy=0$ (see \cite{V}, Subection 1.5). Thus, since $\dim(D)=\dim(\bold 1\oplus \bold 1)=0$, we get 
that $\Dim_+(Y)=0$. Likewise, since $\wedge D=SD$, we have $\wedge^n D=D$ for odd $n$ and $\wedge^nD=\bold 1\oplus \bold 1$ for even $n$, so $\Dim_-(Y)=0$. On the other hand, $\Dim_\pm(X)=\Dim_\pm(Z)=1$, 
so $\Dim_\pm(Y)\ne \Dim_{\pm}(X)+\Dim_{\pm}(Z)$.  
\end{example} 

\begin{remark}\label{Miller} The symmetric category $\Rep D$ appeared 50 years ago in algebraic topology. 
Namely, it is explained in \cite{ATo}, Corollary 7.8 that the K-theory of a topological space with coefficients in $\Bbb Z/2$ 
is a $\Bbb Z/2$-graded commutative algebra in $\Rep D$, with $d$ of degree $1$ (the Bockstein homomorphism), and in fact 
it does not admit (universal) multiplications which are commutative in the usual sense (see also \cite{AY}, (2.17)). 
\end{remark} 

\begin{question}\label{filt} Suppose $p>2$, and $Y$ has a finite filtration. Are the natural epimorphisms 
$\phi_+: S({\rm gr} Y)\to {\rm gr}SY$, $\phi_-: \wedge ({\rm gr} Y)\to {\rm gr}\wedge Y$ isomorphisms?  
\end{question} 

Clearly, for this question it suffices to consider 2-step filtrations. 
By Proposition \ref{prope2}, a positive answer to Question \ref{filt} 
would imply that $\Dim_\pm$ are additive on exact sequences. 

Example \ref{chara2} shows that the answer is ``no" for $p=2$. However, the following proposition shows that 
this type of counterexamples does not exist for $p>2$.  

\begin{proposition} Let $H$ be a cotriangular connected Hopf algebra over a field $\kk$ of characteristic $p>2$ (i.e., every simple $H$-comodule is trivial, and the category of comodules is symmetric). Then for any $\Bbb Z_+$-filtered finite dimensional $H$-comodule $Y$, the natural surjections
$\phi_+,\phi_-$ are isomorphisms. 
\end{proposition} 

\begin{proof} Let us prove the statement for $\phi_+$; for $\phi_-$ the proof is similar. 

First of all, it suffices to prove the statement for any refinement of 
the filtration on $Y$. In particular, we may take a maximal refinement, 
in which the successive quotients are trivial 1-dimensional modules, i.e.,  $Y_0:={\rm gr}Y$ is a vector space with a trivial action of $H$. 

Let ${\rm Rees}(Y):=\prod_{j\ge 0}\hbar^j F_jY$. Then $S({\rm Rees}(Y))$ is a formal deformation of $SY_0=\kk[x_1,...,x_n]$
as a quadratic algebra. Our job is to show that this deformation is flat. By Drinfeld's ``Koszul deformation principle" (\cite{PP}, Introduction, p. ix), it suffices to check this in degrees $2,3$.  If $p>3$, then the action of $S_n$ on $Y_0^{\otimes n}$ is the same as in $Y^{\otimes n}$ for $n=2,3$ (as representations of $S_n$ are semisimple and hence deformation-theoretically rigid), so the condition is clearly satisfied, and we get flatness, as desired.

Consider now the case $p=3$. Then the same argument works  
in degree $2$, and it remains to analyze degree $3$.

We claim that the deformation of the $S_3$-action from $Y_0^{\otimes 3}$ to $Y^{\otimes 3}$ is still trivial. To show this, it suffices to check that
$H^1(S_3,{\rm End}(Y_0^{\otimes 3}))=0$. For this, it is enough to show that for any vector space $W$, we have $H^1(S_3,W^{\otimes 3})=0$
(then we can take $W={\rm End}(Y_0)$). To prove this, let $w_i$ be a basis of $W$. Using this basis, we can decompose 
$W^{\otimes 3}$ as an $S_3$-module into three kinds of modules:

1) The span of the orbit of $w_i\otimes w_j\otimes w_m$ for $i,j,m$ distinct: the regular representation of $S_3$.
This gives the zero cohomology by the Shapiro lemma.

2) The span of the orbit of $w_i\otimes w_i\otimes w_j$ for $i,j$ distinct: the induced module from the trivial module over $S_2$. So again by the Shapiro lemma we get that the cohomology in question is $H^1(S_2,\kk)=0$ (as $p=3$).

3) The span of $w_i\otimes w_i\otimes w_i$. This gives the cohomology $H^1(S_3,\kk)=
{\rm Hom}(S_3,\kk)=0$, as ${\rm char}(\kk)=3$. We are done.
\end{proof} 

\begin{remark} 1. Note that this does not work in characteristic $2$,
even in degree $2$, since for ${\rm char}(\kk)=2$, $H^1(S_2,\kk)=\kk$ is not zero.

2. This proof extends verbatim to the case when simple $H$-comodules factor through a commutative 
Hopf subalgebra $H_0\subset H$ (with trivial triangular structure), i.e. $H_0={\mathcal O}(G)$, where $G$ is an affine group scheme.  

3. We can extend this further to the setting as above, except that $H$ is a Hopf superalgebra, $G$ is an affine supergroup scheme, 
and ${\mathcal O}(G)$ has a triangular structure defined by some central element $u\in G$ of order 2. 
In this case, the proof is the same, except we have to show for $p=3$ that $H^1(S_3,{\rm End}_{\rm even}(Y_0^{\otimes 3}))=0$,
where the action of $S_3$ on ${\rm End}_{\rm even}(Y_0^{\otimes 3})$ takes into account the signs. So case (1) above is the same,
in case (3) we must have all $w_i$ even, 
so we get $H^1(S_3,\kk)=0$, and in case (2)
we get two subcases:

(2a) $w_i$ is even; then we get the induced module from
the trivial $S_2$-module, so we get $H^1(S_2,\kk)=0$; and

(2b) $w_i$ is odd; then we get the induced module from the sign character $\kk_-$ of $S_2$, and
we get $H^1(S_2,\kk_-)$, which is again zero.
\end{remark}

\subsection{Compatibility with the $\lambda$-ring structure} 
Recall now that a pre-$\lambda$-ring is a commutative unital ring with additional operations $\lambda^i$
satisfying the following axioms (see e.g. \cite{AT,Y}). 
\begin{enumerate}
\item $\lambda^0(x) =1$
\item  $\lambda^1(x) = x$
\item $\lambda^n(1) = 0 $ for $n > 1$
\item $\lambda^n(x+y) = \sum_{i+j=n} \lambda^i(x)\lambda^j(y)$
\end{enumerate}

It is an easy exercise to verify that in a symmetric tensor category
$\mathcal{C}$, if for every short exact sequence $0\to X\to Y\to Z\to
0$ the natural surjective map $\psi: \wedge(X\oplus Z)=\wedge({\rm
  gr}Y)\to {\rm gr}\wedge Y$ is an isomorphism (as in  
Proposition \ref{prope2}) then the Grothendieck ring ${\rm Gr}(\mathcal{C})$
inherits the structure of a pre-$\lambda$-ring with the structure maps
$\lambda^i([X])=[\wedge^iX]$.

A pre-$\lambda$-ring is a $\lambda$-ring if it satisfies two additional (families of) axioms encoding plethysm:

\begin{enumerate}
\setcounter{enumi}{4}
\item $\lambda^n(xy) = P_n( \lambda^1(x),\lambda^2(x), \dots ,\lambda^n(x),\lambda^1(y),\lambda^2(y),\dots, \lambda^n(y))$

\item $\lambda^m(\lambda^n(x)) = P_{m,n}( \lambda^1(x),\lambda^2(x), \dots ,\lambda^{mn}(x))$

\end{enumerate}
where $P_n$ and $P_{n,m}$ are certain universal polynomials with integer coefficients (see \cite{AT,Y}). 

For instance, $\Bbb Z$ and $\Bbb Z_p$ are $\lambda$-rings, 
with $\lambda^i(x)=\binom{x}{i}$. Also, the Grothendieck ring of any symmetric tensor category over a field of characteristic zero 
is a $\lambda$-ring, with $\lambda^i=\wedge^i$ (this follows from the fact that this holds for the category of Schur functors, i.e. 
that the ring $\Lambda$ of symmetric functions over $\Bbb Z$ is a $\lambda$-ring). 

\begin{proposition}\label{lambdarings}
Suppose that the Grothendieck ring ${\rm Gr}(\C)$ is a $\lambda$-ring, with $\lambda^i([X])=[\wedge^iX]$ for all $X\in \C$. 
Then $\Dim_-: {\rm Gr}(\C)\to \Bbb Z_p$ is a homomorphism 
of $\lambda$-rings. In particular, $\Dim_-(X\otimes Y)=\Dim_-(X)\Dim_-(Y)$ and $\Dim_-(\wedge^i X)=\binom{\Dim_-(X)}{i}$.    
\end{proposition} 

\begin{proof} Recall that the ring $W(\mathbb{F}_p)$ of big Witt vectors of $\mathbb{F}_p$ is isomorphic to $1 + z\mathbb{F}_p[[z]]$ as an abelian group, and has a $\lambda$-ring structure with the following adjunction property:  if $R$ is a $\lambda$-ring and $f: R \to \mathbb{F}_p$ is a map of rings, then the map $F:= 1+\sum_{j>0} (f\circ \lambda^j)z^j$ is a map of $\lambda$-rings from $R$ to $W(\mathbb{F}_p)$, which, when composed with the ``coefficient of $z$'' map of rings from $W(\mathbb{F}_p)$ to $\mathbb{F}_p$, gives $f$.

By Lemma \ref{dimfp} we may think of the usual categorical dimension as a map of rings $\dim: {\rm Gr}(\mathcal{C}) \to \mathbb{F}_p$. Applying the above adjunction to this gives a map of $\lambda$-rings $\Dim: {\rm Gr}(\mathcal{C}) \to W(\mathbb{F}_p)$.  Explicitly this map sends an object $X$ to $\sum_j \dim(\wedge^jX)z^j = (1+z)^{\Dim_-(X)}$ by Theorem \ref{pdimexists}. 

The copy of $\mathbb{Z}_p$ in $W(\mathbb{F}_p)$ given by the inclusion $\alpha \to (1+z)^\alpha$,   was characterized by Elliot  (\cite{El} Proposition 9.3) as the largest sub-$\lambda$-ring of $W(\mathbb{F}_p)$ for which the Adams operations act by the identity, which implies that the $\lambda$-operations are given by the usual binomial formulas there. We see that $\Dim$ lands inside this sub-$\lambda$-ring and coincides with $\Dim_-$ under its identification with $\mathbb{Z}_p$.

In particular this implies that $\Dim_-$ is a map of $\lambda$-rings from ${\rm Gr}(\mathcal{C})$ to $\mathbb{Z}_p$ (equipped with the binomial $\lambda$-ring structure) and hence satisfies $\Dim_-(X\otimes Y)=\Dim_-(X)\Dim_-(Y)$ and $\Dim_-(\wedge^i X)=\binom{\Dim_-(X)}{i}$, as desired.

 \end{proof} 
 
This proposition applies to classical categories (supergroup representations, etc.) as well as to $\Rep(S_{\mathcal{U}})$, $\Rep(GL_{\mathcal{U}})$.  However, it is possible for both the assumption and conclusion of Proposition \ref{lambdarings} to fail in any characteristic $p\ge 5$. 

To give an example, consider the Verlinde category ${\mathcal{C}}={\rm Ver}_p$, see \cite{O} and references therein. 
This is the quotient of the category of representations of $\Bbb Z/p\Bbb Z$ 
by negligible morphisms. In ${\rm Ver}_p$ we have an object $L_j$ coming from the $j$-dimensional Jordan block representation of $\Bbb Z/p\Bbb Z$ ($1\le j\le p-1$). 
Let $V=L_2$. Then $L_2$ satisfies the equation $P_{p-1}(L_2)=0$, where $P_n$ is the $n$-th ultraspherical polynomial defined by the formula 
$P_n(q+q^{-1})=\sum_{i=0}^n q^{n-2i}$. We claim that $P_{p-1}$ has no roots in $\Bbb Q_p$. Indeed, the roots of $P_p$ are of the form $\zeta+\zeta^{-1}$, 
where $\zeta$ is a root of unity of order $2p$. If such a root were in $\Bbb Q_p$ then $\zeta$ would have been in a quadratic extension 
of $\Bbb Q_p$; but $\Bbb Q_p(\zeta^2)$ is well known to be a ramified extension of degree $p-1>2$.  
This shows that there can't even be a ring homomorphic lift of the dimension map to $\mathbb{Z}_p$. 
In particular, $\Dim_-$ fails to be multiplicative in this example, and hence ${\rm Gr}(\mathcal{C})$ is not (naturally) a $\lambda$-ring. 

\begin{example} Let $p=5$. Then $X=L_3$ satisfies the quadratic equation $X^2-X-1=0$ in ${\rm Gr}(\mathcal{C})$. This equation has no roots in $\Bbb Z_5$, so  
there is no homomorphisms from ${\rm Gr}(\mathcal{C})$ to $\Bbb Z_5$. 
\end{example} 

\subsection{The Koszul complex and equality of $p$-adic dimensions}
For supervector spaces both $\Dim_+$ and $\Dim_-$ recover the usual notion of dimension of a supervector space, and in particular we have 
\begin{equation}\label{equali}
\Dim_+(X)=\Dim_+(X^*)=\Dim_-(X)=\Dim_-(X^*).
\end{equation} 
Therefore, we also get these equalities for any symmetric tensor category admitting a fiber functor to supervector spaces (i.e. supergroup representations) as well as 
interpolations of such categories (i.e., Deligne categories 
$ \Rep(S_{\mathcal{U}})$, $\Rep(GL_{\mathcal{U}})$). However, in a general symmetric tensor category, \eqref{equali} can fail in any characteristic $p\ge 5$. 

To give an example, consider again the category 
${\rm Ver}_p$. It is easy to see that $\wedge^2 V=L_1=\bold 1$ and $\wedge ^iV=0$ for all $i>2$. Also, $\dim(L_j)=j$. Thus, $\Dim_-(V)=2$. 
At the same time, $SV=L_1\oplus L_2\oplus...\oplus L_{p-1}$. Hence, 
$$
\sum_{j=0}^\infty \dim(S^jV)z^j=1+2z+...+(p-1)z^{p-2}=(1-z)^{-2}(1-z^p).
$$
Thus, $\Dim_+(V)=2-p$, and $\Dim_+(V)\ne \Dim_-(V)$.

On the other hand, in ${\rm Ver}_p$ we still have 
$\Dim_+(X)=\Dim_+(X^*)$ and $\Dim_-(X)=\Dim_-(X^*)$.
In fact, we don't know if these equalities can ever fail;
see Question \ref{dimdual} below.  

It is therefore an interesting question when all or some of the equalities 
\eqref{equali} hold. The following two propositions give sufficient conditions for this.

To state the first proposition, recall that to any object $X\in {\mathcal{\C}}$ we can attach its {\it Koszul complex} $K^\bullet(X)=SX\otimes \Lambda^\bullet X$, 
with the usual Koszul differential $\partial : K^i(X)\to K^{i-1}(X)$. Namely, we have the multiplication morphism $X^*\otimes \wedge^{i-1} X^*\to \wedge^i X^*$, 
which after dualization defines the morphism $\mu: X^*\otimes \Lambda^i X\to \Lambda^{i-1}X$. Then 
the differential $\partial$ is given by the formula 
$$
\partial=(m\otimes \mu)\circ P_{23}\circ ({\rm coev}_X\otimes 1_{SX}\otimes 1_{\Lambda X}),
$$
where ${\rm coev}_X: \bold 1\to X\otimes X^*$ is the coevaluation, $P_{23}$ is the permutation of the second and the third factor, 
and $m: X\otimes SX\to SX$ is the multiplication. 

The Koszul complex is a $\Bbb Z_+$-graded complex under diagonal grading. By a classical result of commutative algebra, in the category of vector spaces, this complex is exact. 
More precisely, it is exact in positive homological degrees and has homology equal to $\bold 1$ in degree $0$; i.e., it is a resolution of the augmentation module by free $SX$-modules. 
Note also that for vector spaces $\Lambda X=\wedge X$, so the Koszul complex may be defined as $K^\bullet(X)=SX\otimes \wedge^\bullet X$,
the differential graded algebra generated by the superspace $X\oplus \Pi X$ with differential $d(x,\xi)=(0,x)$ (where $\Pi$ is the change of parity functor). 

Also, if $i+j=n$ then 
$$
S^i\Pi X\otimes \Lambda^j \Pi X=\Pi^n (\wedge^i X\otimes \Bbb S^j X)=\Pi^n (S^jX^*\otimes \Lambda^iX^*)^*.
$$
Thus, the homogeneous components of $K^\bullet(X)$ are dual to those of $K^\bullet(\Pi X^*)$ up to parity change.
This implies that the Koszul complex is exact in the category ${\rm Supervec}$, and hence 
for representation categories of supergroups and their interpolations (Deligne categories). 

Also, it is easy to show that in any symmetric tensor category over $\kk$, the Koszul complex is exact if ${\rm char}(\kk)=0$ (as it is exact in the category of Schur functors), 
and is exact in diagonal degrees $<p$ for characteristic $p$. However, the following proposition implies that in general, the complex $K^\bullet(X)$ can fail to be exact in 
any characteristic $p\ge 5$. 

\begin{proposition} $\Dim_-(X^*) = \Dim_+(X)$ if and only if the Euler characteristic of the Koszul complex $K^\bullet(X)$ is zero outside of diagonal degree zero. 
\end{proposition}

\begin{proof}
It is immediate from the definition of $\Dim_\pm$ that $\Dim_-(X^*) = \Dim_+(X)$ iff
$$\big(\sum_{j\ge 0}(-1)^j\dim(S^jX)z^j\big)\big(\sum_{j\ge 0}\dim(\Lambda^j X)z^j\big) = 1$$
Taking the coefficient of $z^n$ for $n >0$ tells us that $$\sum_j (-1)^j \dim(S^jX)\dim(\Lambda^{n-j} X) = 0$$ and one can easily recognize the left hand side as being the Euler characteristic of the degree $n$ part of the Koszul complex.
\end{proof}

Next, we will give a  sufficient condition for the equality $\Dim_+(X)=\Dim_-(X)$ which is related to 
Proposition \ref{lambdarings}.

\begin{proposition}\label{coin} 
Suppose $char(\kk) > 2$, $\C$ contains ${\rm Supervec}$, and $\Dim_-$ is multiplicative on $\C$ (e.g., the Grothendieck ring ${\rm Gr}(\C)$ is a $\lambda$-ring 
with the usual operations). Then $\Dim_-(X) = \Dim_+(X)$ for all $X\in {\mathcal{C}}$.
\end{proposition}

\begin{proof} Let $\kk_-$ be the nontrivial invertible object in ${\rm Supervec}\subset {\mathcal C}$. 
We have that $\Dim_-(\kk_-)=\Dim_+(\kk_-)=-1$, and moreover $\kk_-$ exchanges the roles of symmetric and exterior powers in the following sense:
$$\wedge^j(\kk_- \otimes X) = \kk_-^{\otimes j}\otimes S^jX$$

Now, since $\Dim_-$ is multiplicative, we have:
\begin{equation*}
\begin{split}
(1+z)^{-\Dim_-(X)}  & = (1+z)^{\Dim_-(\kk_- \otimes X)} \\
 & = \sum_j \dim(\wedge^j(\kk_-\otimes X))z^j  \\
 & = \sum_j \dim( \kk_-^{\otimes j}\otimes S^jX)z^j \\
 & = \sum_j (-1)^j \dim(S^jX)z^j\\
 & = (1 - (-z))^{-\Dim_+(X)} = (1+z)^{-\Dim_+(X)}
\end{split}
\end{equation*}
In particular we see that $\Dim_-(X) = \Dim_+(X)$, as desired.
\end{proof}

\begin{remark} If ${\mathcal{D}}$ is any tensor category over $\kk$ then Proposition \ref{coin} may be applied to the category ${\mathcal{C}}:=
{\rm Supervec}\boxtimes {\mathcal{D}}$, provided that $\Dim_-$ is multiplicative on this category.   
\end{remark} 

\subsection{Categorical Brauer characters} 
One can use $p$-adic dimension to define Brauer characters in the categorical setting. 
Let $G$ be a finite group acting on $X\in \C$. Then for each $g\in G$ and each root of unity $\zeta\in \kk$, we have the generalized eigenobject $X(g,\zeta)$, the direct limit of kernels of 
$(g-\zeta)^N|_X$ as $N\to\infty$. Let $K$ be the maximal unramified extension of $\Bbb Q_p$ (obtained by adding all roots of unity of
orders prime to $p$). Then by Hensel's lemma, $\zeta$ has a canonical lift $\hat \zeta$ to $K$. Define 
$$
\chi_X(g)=\sum_\zeta \Dim_+(X(\zeta,g))\hat\zeta\in K.
$$
We may call this the Brauer character of $G$ on $X$. Note that if $X$ is a usual representation then it is the usual Brauer character. 
One can also make a similar definition using $\Dim_-$. 

\subsection{Questions} 

Here are some questions which could be a subject of further research. 
Let $\C$ be a symmetric tensor category over a field $\kk$ of characteristic $p>0$. 
Let $X\in {\mathcal{C}}$.

\begin{question} \label{dimdual}  
1. Do $(S^iX)^*$ and $S^iX^*$ 
always define the same class in the Grothendieck group ${\rm Gr}({\mathcal{C}})$? Same question for $(\wedge^iX)^*$, $\wedge^iX^*$. 

2. Is it true that $\dim(S^iX)=\dim(S^iX^*)$, $\dim(\wedge^iX)=\dim(\wedge^iX^*)$? 
In other words, is it true that $\Dim_+(X)=\Dim_+(X^*)$, $\Dim_-(X)=\Dim_-(X^*)$?
\end{question} 

A positive answer to (1) would imply one to (2). Also, (1) holds for representation categories 
of finite groups, as the objects in question have the same Brauer characters (using Molien's formula). 

\begin{question} Suppose $p=2$ or $p=3$. 

1. Is it true that $\Dim_+=\Dim_-$? 

2. If the maps $\phi_+$, $\phi_-$ of Proposition \ref{prope2} are always bijective for ${\mathcal{C}}$, is it true that $\Dim_\pm$ are 
ring homomorphisms ${\rm Gr}(\C)\to \Bbb Z_p$? is ${\rm Gr}(\C)$ a $\lambda$-ring with the usual exterior power operations in this case? 

3. Is the Koszul complex $K^\bullet(X)$ always exact? 
\end{question} 

\begin{question}
Suppose for some $C\ge 1$ one has ${\rm length}(X^{\otimes n})\le C^n$ for any $n\ge 1$. Does it follow that 
$\Dim_\pm(X)$ are integers? 
\end{question} 

For the next question, let $\lambda$ be a partition of $n$, and $\pi_\lambda$ be the corresponding Specht module over $\kk S_n$. 
Define the Schur functor ${\mathcal{S}}^\lambda X:=\pi_\lambda\otimes_{S_n}X^{\otimes n}$. 
E.g., ${\mathcal{S}}^{(n)}X=S^nX$ and ${\mathcal{S}}^{(1^n)}X=\wedge^n X$.  
Recall that in characteristic zero, there exists a polynomial $P_\lambda$ such that 
$\dim {\mathcal{S}}^\lambda X=P_\lambda(\dim X)$. However, as we have seen, 
in positive characteristic this fails even for the functors $S^i$ and $\wedge^i$, and the situation is more complicated. 
This motivates the following question. 

\begin{question} 
Suppose $\Dim_+=\Dim_-$, and $\Dim_+(X)=\Dim_-(X)=t$. When can one express $\dim {\mathcal{S}}^\lambda X$
in terms of $t$? 
\end{question} 

Note that there are other versions of Schur functors: we can replace the functor 
$\pi_\lambda\otimes_{S_n} ?$ with the functor $\Hom_{S_n}(\pi_\lambda,?)$
(which for $\lambda=(n)$ and $(1^n)$ yields the functors $\Bbb S$ and $\Lambda$, respectively), 
and can also replace $\pi_\lambda$ with $\pi_\lambda^*$. One may also consider the corresponding derived functors. 
The behavior of dimensions under such functors is an interesting open problem.

\end{document}